\documentclass[12pt]{article}%
\usepackage{amssymb}
\usepackage{amsfonts}
\usepackage{amsmath}
\usepackage{graphicx}%
\newtheorem{theorem}{Theorem}

\newtheorem{lemma}[theorem]{Lemma}

\newtheorem{remark}[theorem]{Remark}

\newenvironment{proof}[1][Proof]{\noindent\textbf{#1.} }{\ \rule{0.5em}{0.5em}}
\begin{document}

\author{Dariusz Idczak\\\textit{Faculty of Mathematics and Computer Science}\\\textit{University of Lodz}\\\textit{Banacha 22, 90-238 Lodz, Poland}\\idczak@math.uni.lodz.pl}
\title{On some strengthening\\of the global implicit function theorem\\with an application to a Cauchy problem \\for an integro-differential Volterra system}
\date{}
\maketitle

\begin{abstract}
In the paper, we improve our earlier results concerning the existence,
uniqueness and differentiability of a global implicit function. Some
application to a Cauchy problem for an integro-differential Volterra system of
nonconvolution type, is given.

\end{abstract}

\section{Introduction}

In paper \cite{ISW}, the conditions for a $C^{1}$-mapping
\[
f:X\rightarrow H
\]
where $X$ is a real Banach space, $H$ - real Hilbert space, to be the
diffeomorphism. i.e. conditions guarantying that for any $y\in H$ there exists
a unique solution $x_{y}\in X$ of the equation
\[
f(x)=y
\]
and the mapping%
\[
H\ni y\longmapsto x_{y}\in X
\]
is continuously differentiable, are given. These conditions are the following:

\begin{itemize}
\item[$\cdot$] the Frechet differential $f^{\prime}(x):X\rightarrow H$ is
bijective, for any $x\in X$,

\item[$\cdot$] the functional%
\[
\varphi:X\ni x\longmapsto(1/2)\left\Vert f(x)-y\right\Vert ^{2}\in\mathbb{R}%
\]
satisfies Palais-Smale condition, for any $y\in H$.
\end{itemize}

The method used in the proof is based on the Mountain Pass Theorem due to
Ambrosetti and Rabinovitz (\cite{AmR}). The obtained result is applied to the
Cauchy problem for an integro-differential Volterra system
\begin{align*}
x^{\prime}(t)+\int_{0}^{t}\Phi(t,\tau,x(\tau))d\tau &  =y(t),\ t\in
\lbrack0,1]\text{ a.e.}\\
x(0)=0  &
\end{align*}
where $y\in L^{2}([0,1],\mathbb{R}^{n})$ and $x\in AC^{2}([0,1],\mathbb{R}%
^{n})$.

In a paper \cite{Id} we extend these results to the case of the equation
\[
F(x,y)=0
\]
where
\[
F:X\times Y\rightarrow H
\]
and $X$, $Y$ are real Banach spaces, $H$ - a real Hilbert space. More
precisely, we formulate sufficient conditions for the existence, uniqueness
and continuous differentiability of a global implicit function $y\longmapsto
x_{y}$ determined by the above equation. The obtained global implicit function
theorem is applied to the Cauchy problem%
\begin{align*}
x^{\prime}(t)+\int_{0}^{t}\Phi(t,\tau,x(\tau),u(\tau))d\tau &  =v(t),\ t\in
J\text{ a.e.,}\\
x(0)=0 &
\end{align*}
where $u$, $v\in L^{2}([0,1],\mathbb{R}^{n})$ and $x\in AC^{2}%
([0,1],\mathbb{R}^{n})$.

In the presented paper, we consider separately

\begin{itemize}
\item[$\cdot$] the existence of a global implicit function

\item[$\cdot$] its uniqueness

\item[$\cdot$] its continuous differentiability
\end{itemize}

We show that the assumptions can be slightly weakened with relation to the
mentioned global implicit function theorem. More precisely, in the theorem on
the existence of a global implicit function, we replace continuous
differentiability of $F$ in Frechet sense, with respect to $(x,y)$, by
differentiability of $F$ in Gateaux sense, with respect to $x$. We also
replace bijectivity of differentials $F_{x}(x,y)$ by the condition $F(x,y)\in
F_{x}(x,y)X$ (cf. also Remark \ref{uw1}). In the theorem on the uniqueness of
the global implicit function, we replace continuous differentiability of $F$
with respect to $(x,y)$ by its continuous differentiability with respect to
$x$. Moreover, in theorems on the uniqueness and continuous differentiability
of the global implicit function, bijectivity of $F_{x}(x,y)$ is assumed only
for points $(x,y)$ satisfying equality $F(x,y)=0$ and for the remaining points
$(x,y)$ one assumes that $F_{x}(x,y)\in F_{x}(x,y)X$ (cf. also Remarks
\ref{uw2}, \ref{uw3}). As in \cite{ISW} and \cite{Id}, we use a variational
approach based on the Mountain Pass Theorem. We apply the obtained theorem to
Cauchy problem%
\begin{align*}
x^{\prime}(t)+\int_{0}^{t}\Phi(t,\tau,x(\tau),u(\tau))d\tau &
=f(t,x(t),v(t)),\ t\in J=[0,1]\text{ a.e.,}\\
x(0)=0 &
\end{align*}
where $u\in L^{\infty}(J,\mathbb{R}^{m})$, $v\in L^{\infty}(J,\mathbb{R}^{r}%
)$, $x\in AC^{2}([0,1],\mathbb{R}^{n})$ (system of the above type is studied
in \cite{Lakshmi}).

Our paper consists of two parts. In the first part, we derive three theorems:
on the existence of a global implicit function, on the uniqueness of this
function as well as on the continuous differentiability of it. The second part
is devoted to some application. We study an integro-differential Cauchy
problem for Volterra system of general - nonconvolution type (\cite{Lakshmi})
with two functional parameters $u$, $v$ that are involved nonlinearly. Problem
of such a type but with the term containing $v$ replaced by $v$, was
investigated in \cite{Id}. We obtain existence and uniqueness as well as the
continuous differentiability of the mapping describing dependence of solutions
on parameters. Differential of this mapping is given, too.

\section{Existence of a global implicit function}

Let $X$ be a real Banach space and $I:X\rightarrow\mathbb{R}$ - a functional
differentiable in Gateaux sense. We say that $I$ satisfies
\textit{Palais-Smale (PS) condition} if any sequence $(x_{m})$ satisfying conditions:

\begin{itemize}
\item[$\bullet$] $\left\vert I(x_{m})\right\vert \leq M$ for all
$m\in\mathbb{N}$ and some $M>0,$

\item[$\bullet$] $I^{\prime}(x_{m})\longrightarrow0,$
\end{itemize}

\noindent admits a convergent subsequence ($I^{\prime}(x_{m})$ denotes the
Gateaux differential of $I$ at $x_{m}$). A sequence $(x_{m})$ satisfying the
above conditions is called the (PS) sequence.

A point $x^{\ast}\in X$ is called the \textit{critical point} of $I$ if
$I^{\prime}(x^{\ast})=0$. In such a case $I(x^{\ast})$ is called the
\textit{critical value} of $I$.

In \cite[Corollary 3.3]{Maw} the following theorem is proved.

\begin{theorem}
\label{Mawhin}Let $X$ be a real Banach space. If $I:X\rightarrow\mathbb{R}$ is
lower semicontinuous, bounded below and differentiable in Gateaux sense
functional satisfying (PS) condition, then there exists a critical point
$x^{\ast}$ of $I$.
\end{theorem}

Using the above theorem we obtain

\begin{theorem}
\label{existence}Let $X$ be a real Banach space, $Y$ - a nonempty set, $H$ - a
real Hilbert space. If $F:X\times Y\rightarrow H$ is differentiable with
respect to $x\in X$ in Gateaux sense and

\begin{itemize}
\item[$\bullet$] $F(x,y)\in F_{x}(x,y)X$ for any $(x,y)\in X\times Y$
($F_{x}(x,y)$ denotes the Gateaux differential of $F$ at $(x,y)$ with respect
to $x$)

\item[$\bullet$] the functional
\begin{equation}
\varphi:X\ni x\longmapsto(1/2)\left\Vert F(x,y)\right\Vert ^{2}\in\mathbb{R}
\label{fi}%
\end{equation}
is lower semicontinuous and satisfies (PS) condition for any $y\in Y$,
\end{itemize}

\noindent then, for any $y\in Y$, there exists $x_{y}\in X$ such that
$F(x_{y},y)=0$.
\end{theorem}

\begin{proof}
Let us fix a point $y\in Y$. Functional $\varphi$, being a superposition of
the mapping $(1/2)\left\Vert \cdot\right\Vert ^{2}$ differentiable in Frechet
sense on $H$ and the mapping $F(\cdot,y)$ differentiable in Gateaux sense on
$X$, is differentiable in Gateaux sense on $X$ and its Gateaux differential
$\varphi^{\prime}(x)$ at $x\in X$ is given by%
\[
\varphi^{\prime}(x)h=\left\langle F(x,y),F_{x}(x,y)h\right\rangle
\]
for $h\in X$. Moreover, $\varphi$ is bounded below and, by assumption, lower
semicontinuous and satisfies (PS) condition. So, by Theorem \ref{Mawhin},
there exists a point $x_{y}\in X$ such that%
\[
\left\langle F(x_{y},y),F_{x}(x_{y},y)h\right\rangle =0
\]
for $h\in X$. Since $F(x_{y},y)\in F_{x}(x_{y},y)X$, $F(x_{y},y)=0$.
\end{proof}

\begin{remark}
The assumption on lower semicontinuity of $\varphi$ can be replaced by a more
restrictive one but concerning directly $F$, namely - continuity of $F$ with
respect to $x$.
\end{remark}

\begin{remark}
\label{uw1}The assumption "$F(x,y)\in F_{x}(x,y)X$ for any $(x,y)\in X\times
Y$" can be replaced by the following one "$F(x,y)\in F_{x}(x,y)X$ for any
$(x,y)\in X\times Y$ such that $\varphi^{\prime}(x)=0$ with $\varphi$
determined by $y$".
\end{remark}

\begin{remark}
It is known (cf. \cite[Corollary 3.4]{Maw}) that if a functional
$I:X\rightarrow\mathbb{R}$ ($X$ - a Banach space) is lower semicontinuous,
bounded below, differentiable in Gateaux sense, has a bounded minimizing
sequence and satisfies the weak (PS) condition (i.e. any bounded (PS) sequence
has a convergent subsequence), then $I$ has a critical point. So, (PS)
condition in Theorem \ref{existence} can be replaced by the following one:
$\varphi$ has a bounded minimizing sequence and satisfies the weak (PS) condition.
\end{remark}

\section{Uniqueness of a global implicit function}

Let $d\neq0$ be a point of $X$ (a real Banach space). By $W_{d}$ we denote the
set%
\[
W_{d}=\{U\subset X;\ U\text{ is open, }0\in U\text{ and }d\notin\overline
{U}\}.
\]

\noindent We have (\cite{AmR}, \cite{Rabin})

\begin{theorem}
[Mountain Pass Theorem]Let $I:X\rightarrow\mathbb{R}$ be a functional which is
continuously differentiable in Gateaux (equivalently, in Frechet) sense,
satisfies (PS) condition and $I(0)=0$. If there exist constants $\rho$,
$\alpha>0$ such that $I\mid_{\partial B(0,\rho)}\geq\alpha$ and $I(e)\leq0$
for some $e\in X\setminus\overline{B(0,\rho)}$, then%
\[
b:=\underset{U\in W_{e}}{\sup}\underset{u\in\partial U}{\inf}I(u)
\]
is the critical value of $I$ and $b\geq\alpha$ ( \footnote{It is known (cf.
\cite{IT}) that if a mapping is continuously differentiable at a point in
Gateaux sense then it is differentiable at this point in Frechet sense and
both differentials coincide.}).
\end{theorem}

Using the above theorem we can prove

\begin{theorem}
\label{uniqueness}Let $X$ be a real Banach space, $Y$ - a nonempty set, $H$ -
a real Hilbert space. If $F:X\times Y\rightarrow H$ is continuously
differentiable with respect to $x\in X$ in Gateaux (equivalently, in Frechet)
sense and

\begin{itemize}
\item[$\bullet$] $F_{x}(x,y):X\rightarrow Y$ is bijective for any $(x,y)\in
X\times Y$ such that $F(x,y)=0$ and $F(x,y)\in F_{x}(x,y)X$ for the remaining
$(x,y)\in X\times Y$

\item[$\bullet$] the functional $\varphi$ given by (\ref{fi}) satisfies (PS)
condition for any $y\in Y$,
\end{itemize}

\noindent then, for any $y\in Y$, there exists a unique $x_{y}\in X$ such that
$F(x_{y},y)=0$.
\end{theorem}

\begin{proof}
Let us fix $y\in Y$. From Theorem \ref{existence} it follows that there exists
a point $x_{y}\in X$ such that $F(x_{y},y)=0$. Let us suppose that there exist
$x_{1}$, $x_{2}\in X$, $x_{1}\neq x_{2}$, such that $F(x_{1},y)=F(x_{2},y)=0$.
Put $e=x_{2}-x_{1}$ and%
\[
g(x)=F(x+x_{1},y)
\]
for $x\in X$. Of course,
\[
g(x)=g^{\prime}(0)x+o(x)=F_{x}^{\prime}(x_{1},y)x+o(x)
\]
for $x\in X$, where $o(x)/\left\Vert x\right\Vert _{X}\rightarrow0$ in $H$
when $x\rightarrow0$ in $X$. So,%
\[
\beta\left\Vert x\right\Vert _{X}\leq\left\Vert F_{x}^{\prime}(x_{1}%
,y)x\right\Vert _{H}\leq\left\Vert g(x)\right\Vert _{H}+\left\Vert
o(x)\right\Vert _{H}\leq\left\Vert g(x)\right\Vert _{H}+(1/2)\beta\left\Vert
x\right\Vert _{X}%
\]
for sufficiently small $\left\Vert x\right\Vert _{X}$ and some $\beta>0$
(existence of such an $\beta$ follows from the bijectivity of $F_{x}^{\prime
}(x_{1},y)$). Thus, there exists $\rho>0$ such that
\[
(1/2)\beta\left\Vert x\right\Vert _{X}\leq\left\Vert g(x)\right\Vert _{H}%
\]
for $x\in\overline{B(0,\rho)}$. Without loss of the generality one may assume
that $\rho<\left\Vert e\right\Vert _{X}$. Put%
\[
\psi(x)=(1/2)\left\Vert g(x)\right\Vert _{H}^{2}=(1/2)\left\Vert
F(x+x_{1},y)\right\Vert _{H}^{2}=\varphi(x+x_{1})
\]
for $x\in X$. Of course, $\psi$ is continuously differentiable on $X$ in
Gateaux sense and%
\[
\psi^{\prime}(x)=\varphi^{\prime}(x+x_{1}).
\]
Consequently, since $\varphi$ satisfies (PS) condition, $\psi$ has this
property, too. Moreover, $\psi(0)=\psi(e)=0$, $e\notin\overline{B(0,\rho)}$
and $\psi(x)\geq\alpha$ for $x\in\partial B(0,\rho)$ with $\alpha
=(1/8)\beta^{2}\rho^{2}>0.$

Thus, the Mountain Pass Theorem implies that $b=\underset{U\in W_{e}}{\sup
}\underset{x\in\partial U}{\inf}\psi(x)$ is a critical value of $\psi$ and
$b\geq\alpha$, i.e. there exists a point $x^{\ast}\in X$ such that
$\psi(x^{\ast})=b>0$ and
\[
\psi^{\prime}(x_{\ast})h=\left\langle F(x^{\ast}+x_{1},y),F_{x}(x^{\ast}%
+x_{1},y)h\right\rangle =0
\]
for $h\in X$. The first condition means that $F(x^{\ast}+x_{1},y)\neq0$. The
second one and relation $F(x^{\ast}+x_{1},y)\in F_{x}^{\prime}(x^{\ast}%
+x_{1},y)X$ imply that $F(x^{\ast}+x_{1},y)=0$. The obtained contradiction
completes the proof.
\end{proof}

\begin{remark}
\label{uw2}The assumption "$F(x,y)\in F_{x}(x,y)X$ for the remaining $(x,y)\in
X\times Y$" can be replaced by the following one "$F(x,y)\in F_{x}(x,y)X$ for
the remaining $(x,y)\in X\times Y$ such that $\varphi^{\prime}(x)=0$ with
$\varphi$ determined by $y$".
\end{remark}

When $X=\mathbb{R}^{n}$, (PS) condition imposed on $\varphi$ can be replaced
by the following (equivalent) one: \textquotedblright$\varphi$ is coercive,
i.e. $\varphi(x)\rightarrow\infty$ when $\left\vert x\right\vert
\rightarrow\infty$\textquotedblright. It follows from the following two lemmas.

\begin{lemma}
If a functional $I:\mathbb{R}^{n}\rightarrow\mathbb{R}$ is coercive, then it
satisfies (PS) condition.
\end{lemma}

\begin{proof}
The assertion follows immediately from the boundedness of relatively compact
sets in $\mathbb{R}^{n}$.
\end{proof}

\begin{lemma}
If $X$ is a real Banach space and a functional $I\in C^{1}(X,\mathbb{R})$ is
bounded below and satisfies (PS) condition, then it is coercive.
\end{lemma}

\begin{proof}
Let us suppose that $I$ is not coercive. So, there exists a sequence $(x_{n})$
such that $\left\Vert x_{n}\right\Vert \rightarrow\infty$ and the sequence
$(I(x_{n}))$ is upper bounded. Of course, it is bounded below, too. Thus,
$c:=\underset{\left\Vert x\right\Vert \rightarrow\infty}{\lim\inf}%
I(x)\in\mathbb{R}$ and using \cite[Corollar 2.7]{Willem} (\footnote{If $I\in
C^{1}(X,\mathbb{R})$ is bounded below and any sequence $(x_{n})$ such that%
\[
I(x_{n})\rightarrow c:=\underset{\left\Vert u\right\Vert \rightarrow\infty
}{\lim\inf}I(x_{n}),\ I^{\prime}(x_{n})\rightarrow0
\]
is bounded, then $I$ is coercive.}) we obtain existence of a sequence
$(x_{n})$ such that $I(x_{n})\rightarrow c$, $I^{\prime}(x_{n})\rightarrow0$
and $\left\Vert x_{n}\right\Vert \rightarrow\infty$. It contradicts (PS) condition.
\end{proof}

\section{Global implicit function theorem}

From Theorems \ref{existence}, \ref{uniqueness} and classical local implicit
function theorem we immediately obtain the following global implicit function theorem.

\begin{theorem}
\label{gift2}Let $X$, $Y$ be real Banach spaces, $H$ - a real Hilbert space.
If $F:X\times Y\rightarrow H$ is continuously differentiable in Gateaux
(equivalently, in Frechet) sense with respect to $(x,y)\in X\times Y$ and

\begin{itemize}
\item[$\bullet$] differential $F_{x}(x,y):X\rightarrow H$ is bijective for any
$(x,y)\in X\times Y$ such that $F(x,y)=0$ and $F(x,y)\in F_{x}(x,y)X$ for the
remaining $(x,y)\in X\times Y$

\item[$\bullet$] the functional $\varphi$ given by (\ref{fi}) satisfies the
(PS) condition for any $y\in Y$,
\end{itemize}

\noindent then there exists a unique function $\lambda:Y\rightarrow X$ such
that $F(\lambda(y),y)=0$ for any $y\in Y$ and this function is continuously
differentiable in Gateaux (equivalently, in Frechet) sense on $Y$ with
differential $\lambda^{\prime}(y)$ at $y$ given by
\begin{equation}
\lambda^{\prime}(y)=-[F_{x}(\lambda(y),y)]^{-1}\circ F_{y}(\lambda(y),y).
\label{wzor}%
\end{equation}

\end{theorem}

\begin{proof}
Of course, it is sufficient to put $\lambda(y)=x_{y}$ where $x_{y}$ is a
solution to $F(x,y)=0$, given by Theorem \ref{uniqueness}.
\end{proof}

\begin{remark}
\label{uw3}Remark \ref{uw3} is applicable.
\end{remark}

\section{An application}

Let us consider the following control system%
\begin{equation}
x^{\prime}(t)+\int_{0}^{t}\Phi(t,\tau,x(\tau),u(\tau))d\tau
=f(t,x(t),v(t)),\ t\in J=[0,1]\text{ a.e.,} \label{cs}%
\end{equation}
where $\Phi:P_{\Delta}\times\mathbb{R}^{n}\times\mathbb{R}^{m}\rightarrow
\mathbb{R}^{n}$ ($P_{\Delta}=\{(t,\tau)\in J\times J;\tau\leq t\}$),
$f:J\times\mathbb{R}^{n}\times\mathbb{R}^{r}\rightarrow\mathbb{R}^{n}$, $x\in
AC_{0}^{2}=AC_{0}^{2}(J,\mathbb{R}^{n})=\{x:J\rightarrow\mathbb{R}^{n}$; $x$
is absolutely continuous, $x(0)=0$, $x^{\prime}\in L^{2}(J,\mathbb{R}^{n})\}$,
$u\in L^{\infty}(J,\mathbb{R}^{m})$, $v\in L^{\infty}(J,\mathbb{R}^{r})$. On
the functions $\Phi$, $f$ we assume that

\begin{itemize}
\item[$\cdot$] $\Phi(\cdot,\cdot,x,u)$ is measurable on $P_{\Delta}$ for any
$x\in\mathbb{R}^{n}$, $u\in\mathbb{R}^{m}$; $\Phi(t,\tau,\cdot,\cdot)$ is
continuously differentiable on $\mathbb{R}^{n}\times\mathbb{R}^{m}$ for
$(t,\tau)\in P_{\Delta}$ a.e.

\item[$\cdot$] there exist constants $c$, $d>0$ and functions $a$, $b\in
L^{2}(P_{\Delta},\mathbb{R}_{0}^{+})$, $\omega\in C(\mathbb{R}_{0}%
^{+},\mathbb{R}_{0}^{+})$ such that%
\[
\left\vert \Phi(t,\tau,x,u)\right\vert \leq a(t,\tau)\left\vert x\right\vert
+b(t,\tau)\omega(\left\vert u\right\vert )
\]%
\[
\left\vert \Phi_{x}(t,\tau,x,u)\right\vert \leq c\omega(\left\vert
x\right\vert )+d\omega(\left\vert u\right\vert )\text{,}%
\]%
\[
\left\vert \Phi_{u}(t,\tau,x,u)\right\vert \leq a(t,\tau)\omega(\left\vert
x\right\vert )+b(t,\tau)\omega(\left\vert u\right\vert )
\]
for $(t,\tau)\in P_{\Delta}$ a.e., $x\in\mathbb{R}^{n}$, $u\in\mathbb{R}^{m}$

\item[$\cdot$] $f(\cdot,x,u)$ is measurable on $J$ for any $x\in\mathbb{R}%
^{n}$, $v\in\mathbb{R}^{r}$; $f(t,\cdot,\cdot)$ is continuously differentiable
on $\mathbb{R}^{n}\times\mathbb{R}^{r}$ for $t\in J$ a.e.

\item[$\cdot$] there exist constants $c_{f}$, $d_{f}>0$ and functions $a_{f}$,
$b_{f}\in L^{2}(J,\mathbb{R}_{0}^{+})$, $\varkappa\in C(\mathbb{R}_{0}%
^{+},\mathbb{R}_{0}^{+})$ such that%
\[
\left\vert f(t,x,v)\right\vert \leq a_{f}(t)\left\vert x\right\vert
+b_{f}(t)\varkappa(\left\vert v\right\vert )
\]%
\[
\left\vert f_{x}(t,x,v)\right\vert \leq c_{f}\varkappa(\left\vert x\right\vert
)+d_{f}\varkappa(\left\vert v\right\vert )
\]%
\[
\left\vert f_{v}(t,x,v)\right\vert \leq a_{f}(t)\varkappa(\left\vert
x\right\vert )+b_{f}(t)\varkappa(\left\vert v\right\vert )
\]
for $t\in J$ a.e., $x\in\mathbb{R}^{n}$, $v\in\mathbb{R}^{r}$

\item[$\cdot$] the inequality
\[
\left\Vert a\right\Vert _{L^{2}(P_{\Delta},\mathbb{R})}+2(\int_{0}^{1}%
(a_{f}(t))^{2}tdt)^{(1/2)}(1+\left\Vert a\right\Vert _{L^{2}(P_{\Delta
},\mathbb{R})})<\sqrt{2}/2
\]
is satisfied.
\end{itemize}

We shall check that the mapping
\[
F:AC_{0}^{2}\times L^{\infty}(J,\mathbb{R}^{m})\times L^{\infty}%
(J,\mathbb{R}^{r})\rightarrow L^{2}(J,\mathbb{R}^{n}),
\]%
\[
F(x,u,v)=x^{\prime}(t)+\int_{0}^{t}\Phi(t,\tau,x(\tau),u(\tau))d\tau
-f(t,x(t),v(t)),
\]
satisfies assumptions of global implicit function theorem with $X=AC_{0}^{2}$,
$Y=L^{\infty}(J,\mathbb{R}^{m})\times L^{\infty}(J,\mathbb{R}^{r})$,
$H=L^{2}(J,\mathbb{R}^{n})$.

In a standard way, one can check that $F$ is continuously differentiable in
Gateaux (equivalently, in Frechet) sense on $AC_{0}^{2}\times L^{\infty
}(J,\mathbb{R}^{m})\times L^{\infty}(J,\mathbb{R}^{r})$ and the mappings%
\[
F_{x}(x,u,v):AC_{0}^{2}\rightarrow L^{2}(J,\mathbb{R}^{n}),
\]%
\[
F_{x}(x,u,v)h=h^{\prime}(t)+\int_{0}^{t}\Phi_{x}(t,\tau,x(\tau),u(\tau
))h(\tau)d\tau-f_{x}(t,x(t),v(t))h(t)
\]%
\[
F_{u,v}(x,u,v):L^{\infty}(J,\mathbb{R}^{m})\times L^{\infty}(J,\mathbb{R}%
^{n})\rightarrow L^{2}(J,\mathbb{R}^{n}),
\]%
\[
F_{u,v}(x,u,v)(f,g)=\int_{0}^{t}\Phi_{u}(t,\tau,x(\tau),u(\tau))f(\tau
)d\tau-f_{v}(t,x(t),v(t))g(t)
\]
are the differentials of $F$ in $x$ and $(u,v)$, respectively. From Appendix
it follows that $F_{x}(x,u,v)$ is \textquotedblright one-one\textquotedblright%
\ and \textquotedblright onto\textquotedblright.

Now, let us fix a function $(u,v)\in L^{\infty}(J,\mathbb{R}^{m})\times
L^{\infty}(J,\mathbb{R}^{r})$ and consider the functional%
\begin{multline*}
\varphi:AC_{0}^{2}\ni x\longmapsto(1/2)\left\Vert F(x,u,v)\right\Vert ^{2}\\
=(1/2)\int\nolimits_{0}^{1}\left\vert x^{\prime}(t)+\int_{0}^{t}\Phi
(t,\tau,x(\tau),u(\tau))d\tau-f(t,x(t),v(t))\right\vert ^{2}\in\mathbb{R}%
\text{.}%
\end{multline*}
It is easy to see that, for any $x\in AC_{0}^{2}$,%
\begin{multline*}
\left\vert \varphi(x)\right\vert \geq(1/2)\left\Vert x\right\Vert _{AC_{0}%
^{2}}^{2}+\int_{0}^{1}x^{\prime}(t)\int_{0}^{t}\Phi(t,\tau,x(\tau
),u(\tau))d\tau dt\\
-\int_{0}^{1}x^{\prime}(t)f(t,x(t),u(t))dt-\int_{0}^{1}f(t,x(t),v(t))\int
_{0}^{t}\Phi(t,\tau,x(\tau),u(\tau))d\tau dt.
\end{multline*}
Let us observe that%
\begin{multline*}
\left\vert \int_{0}^{1}x^{\prime}(t)\int_{0}^{t}\Phi(t,\tau,x(\tau
),u(\tau))d\tau dt\right\vert \leq\int_{0}^{1}\left\vert x^{\prime
}(t)\right\vert (\int_{0}^{t}(a(t,\tau)\left\vert x(\tau)\right\vert \\
+b(t,\tau)\omega(\left\vert u(\tau)\right\vert ))d\tau)dt\\
\leq\int_{0}^{1}\left\vert x^{\prime}(t)\right\vert ((\int_{0}^{t}a^{2}%
(t,\tau)d\tau)^{(1/2)}(\sqrt{2}/2)\left\Vert x\right\Vert _{AC_{0}^{2}}+A%
{\displaystyle\int\nolimits_{0}^{t}}
b(t,\tau)d\tau)dt\\
\leq(\sqrt{2}/2)\left\Vert x\right\Vert _{AC_{0}^{2}}(\int_{0}^{1}\left\vert
x^{\prime}(t)\right\vert ^{2}dt)^{(1/2)}(\int_{0}^{1}(\int_{0}^{t}a^{2}%
(t,\tau)d\tau)dt)^{(1/2)}\\
+A(\int_{0}^{1}\left\vert x^{\prime}(t)\right\vert ^{2}dt)^{(1/2)}(\int
_{0}^{1}(%
{\displaystyle\int\nolimits_{0}^{t}}
b(t,\tau)d\tau)^{2}dt)^{1/2}\\
\leq(\sqrt{2}/2)\left\Vert a\right\Vert _{L^{2}(P_{\Delta},\mathbb{R}%
)}\left\Vert x\right\Vert _{AC_{0}^{2}}^{2}+A\left\Vert b\right\Vert
_{L^{2}(P_{\Delta},\mathbb{R})}\left\Vert x\right\Vert _{AC_{0}^{2}}%
\end{multline*}
where $A=\underset{t\in\lbrack0,1]}{ess\sup}\omega(\left\vert u(t)\right\vert
)$. Also,%
\begin{multline*}
\int_{0}^{1}\left\vert f(t,x(t),u(t))\right\vert ^{2}dt\leq\int_{0}^{1}\left(
a_{f}(t)\left\vert x(t)\right\vert +b_{f}(t)\varkappa(\left\vert
u(t)\right\vert )\right)  ^{2}dt\\
\leq2\int_{0}^{1}\left(  (a_{f}(t))^{2}\left\vert x(t)\right\vert ^{2}%
+(b_{f}(t))^{2}(\varkappa(\left\vert u(t)\right\vert ))^{2}\right)  dt\\
\leq2(\int_{0}^{1}(a_{f}(t))^{2}tdt\left\Vert x\right\Vert _{AC_{0}^{2}}%
^{2}+B\left\Vert b_{f}\right\Vert _{L^{2}(J,\mathbb{R}^{n})}^{2})
\end{multline*}
where $B=\underset{t\in\lbrack0,1]}{ess\sup}(\varkappa(\left\vert
u(t)\right\vert ))^{2}$, so,
\begin{align*}
\left\vert \int_{0}^{1}x^{\prime}(t)f(t,x(t),u(t))dt\right\vert  &
\leq\left\Vert x\right\Vert _{AC_{0}^{2}}(\int_{0}^{1}\left\vert
f(t,x(t),u(t))\right\vert ^{2}dt)^{(1/2)}\\
&  \leq\left\Vert x\right\Vert _{AC_{0}^{2}}(2(\int_{0}^{1}(a_{f}%
(t))^{2}tdt\left\Vert x\right\Vert _{AC_{0}^{2}}^{2}+B\left\Vert
b_{f}\right\Vert _{L^{2}(J,\mathbb{R}^{n})}^{2}))^{(1/2)}\\
&  \leq\left\Vert x\right\Vert _{AC_{0}^{2}}(\sqrt{2}(\int_{0}^{1}%
(a_{f}(t))^{2}tdt)^{(1/2)}\left\Vert x\right\Vert _{AC_{0}^{2}}+\sqrt
{2B}\left\Vert b_{f}\right\Vert _{L^{2}(J,\mathbb{R}^{n})})
\end{align*}
Similarly,%
\begin{multline*}
\left\vert \int_{0}^{1}f(t,x(t),v(t))\int_{0}^{t}\Phi(t,\tau,x(\tau
),u(\tau))d\tau dt\right\vert \\
\leq(\int_{0}^{1}\left\vert f(t,x(t),u(t))\right\vert ^{2}dt)^{(1/2)}(\int
_{0}^{1}\left\vert \int_{0}^{t}\Phi(t,\tau,x(\tau),u(\tau))d\tau\right\vert
^{2}dt)^{(1/2)}\\
\leq(2(\int_{0}^{1}(a_{f}(t))^{2}tdt\left\Vert x\right\Vert _{AC_{0}^{2}}%
^{2}+B\left\Vert b_{f}\right\Vert _{L^{2}(J,\mathbb{R}^{n})}^{2}))^{(1/2)}\\
\times(\int_{0}^{1}((\int_{0}^{t}a^{2}(t,\tau)d\tau)^{(1/2)}(\sqrt
{2}/2)\left\Vert x\right\Vert _{AC_{0}^{2}}+A(\int_{0}^{t}b^{2}(t,\tau
)d\tau))^{(1/2)})^{2}dt)^{(1/2)}\\
\leq(2(\int_{0}^{1}(a_{f}(t))^{2}tdt\left\Vert x\right\Vert _{AC_{0}^{2}}%
^{2}+B\left\Vert b_{f}\right\Vert _{L^{2}(J,\mathbb{R}^{n})}^{2}%
))^{(1/2)}(\left\Vert a\right\Vert _{L^{2}(P_{\Delta},\mathbb{R})}%
^{2}\left\Vert x\right\Vert _{AC_{0}^{2}}^{2}+2A^{2}\left\Vert b\right\Vert
_{L^{2}(P_{\Delta},\mathbb{R})}^{2})^{(1/2)}\\
\leq(\sqrt{2}(\int_{0}^{1}(a_{f}(t))^{2}tdt)^{(1/2)}\left\Vert x\right\Vert
_{AC_{0}^{2}}+\sqrt{2B}\left\Vert b_{f}\right\Vert _{L^{2}(J,\mathbb{R}^{n}%
)})(\left\Vert a\right\Vert _{L^{2}(P_{\Delta},\mathbb{R})}\left\Vert
x\right\Vert _{AC_{0}^{2}}+\sqrt{2}A\left\Vert b\right\Vert _{L^{2}(P_{\Delta
},\mathbb{R})}).
\end{multline*}
Finally,%
\begin{multline*}
\left\vert \varphi(x)\right\vert \geq(1/2)\left\Vert x\right\Vert _{AC_{0}%
^{2}}^{2}-(\sqrt{2}/2)\left\Vert a\right\Vert _{L^{2}(P_{\Delta},\mathbb{R}%
)}\left\Vert x\right\Vert _{AC_{0}^{2}}^{2}-A\left\Vert b\right\Vert
_{L^{2}(P_{\Delta},\mathbb{R})}\left\Vert x\right\Vert _{AC_{0}^{2}}\\
-\left\Vert x\right\Vert _{AC_{0}^{2}}(\sqrt{2}(\int_{0}^{1}(a_{f}%
(t))^{2}tdt)^{(1/2)}\left\Vert x\right\Vert _{AC_{0}^{2}}+\sqrt{2B}\left\Vert
b_{f}\right\Vert _{L^{2}(J,\mathbb{R}^{n})})\\
-(\sqrt{2}(\int_{0}^{1}(a_{f}(t))^{2}tdt)^{(1/2)}\left\Vert x\right\Vert
_{AC_{0}^{2}}+\sqrt{2B}\left\Vert b_{f}\right\Vert _{L^{2}(J,\mathbb{R}^{n}%
)})(\left\Vert a\right\Vert _{L^{2}(P_{\Delta},\mathbb{R})}\left\Vert
x\right\Vert _{AC_{0}^{2}}+\sqrt{2}A\left\Vert b\right\Vert _{L^{2}(P_{\Delta
},\mathbb{R})})\\
=((1/2)-(\sqrt{2}/2)\left\Vert a\right\Vert _{L^{2}(P_{\Delta},\mathbb{R}%
)}-\sqrt{2}(\int_{0}^{1}(a_{f}(t))^{2}tdt)^{(1/2)}\\
-\sqrt{2}(\int_{0}^{1}(a_{f}(t))^{2}tdt)^{(1/2)}\left\Vert a\right\Vert
_{L^{2}(P_{\Delta},\mathbb{R})})\left\Vert x\right\Vert _{AC_{0}^{2}}^{2}\\
-(A+\sqrt{2B}\left\Vert b_{f}\right\Vert _{L^{2}(J,\mathbb{R}^{n})}A\left\Vert
b\right\Vert _{L^{2}(P_{\Delta},\mathbb{R})}+2A\left\Vert b\right\Vert
_{L^{2}(P_{\Delta},\mathbb{R})}(\int_{0}^{1}(a_{f}(t))^{2}tdt)^{(1/2)}\\
+\sqrt{2B}\left\Vert b_{f}\right\Vert _{L^{2}(J,\mathbb{R}^{n})}\left\Vert
a\right\Vert _{L^{2}(P_{\Delta},\mathbb{R})})\left\Vert x\right\Vert
_{AC_{0}^{2}}+\sqrt{2B}\left\Vert b_{f}\right\Vert _{L^{2}(J,\mathbb{R}^{n}%
)}\sqrt{2}A\left\Vert b\right\Vert _{L^{2}(P_{\Delta},\mathbb{R})}%
\end{multline*}
for $x\in AC_{0}^{2}$. This means that $\varphi$ is coercive.

In a standard way, we check that the differential $\varphi^{\prime}(x)$ of
$\varphi$ at $x$ is given by%
\begin{multline*}
\varphi^{\prime}(x)h=\int\nolimits_{0}^{1}(x^{\prime}(t)+\int_{0}^{t}%
\Phi(t,\tau,x(\tau),u(\tau))d\tau-f(t,x(t),v(t)))\\
\times(h^{\prime}(t)+\int_{0}^{t}\Phi_{x}(t,\tau,x(\tau),u(\tau))h(\tau
)d\tau-f_{x}(t,x(t),v(t))h(t))dt
\end{multline*}
for $h\in AC_{0}^{2}$. Consequently, for any $x_{m}$, $x_{0}\in AC_{0}^{2}$,%
\begin{multline*}
\varphi^{\prime}(x_{m})(x_{m}-x_{0})=\int\nolimits_{0}^{1}(x_{m}^{\prime
}(t)+\int_{0}^{t}\Phi(t,\tau,x_{m}(\tau),u(\tau))d\tau-f(t,x_{m}(t),v(t)))\\
((x_{m}^{\prime}(t)-x_{0}^{\prime}(t))+\int_{0}^{t}\Phi_{x}(t,\tau,x_{m}%
(\tau),u(\tau))(x_{m}(\tau)-x_{0}(\tau))d\tau\\
-f_{x}(t,x_{m}(t),v(t))(x_{m}(t)-x_{0}(t)))dt
\end{multline*}%
\begin{multline*}
\varphi^{\prime}(x_{0})(x_{m}-x_{0})=\int\nolimits_{0}^{1}(x_{0}^{\prime
}(t)+\int_{0}^{t}\Phi(t,\tau,x_{0}(\tau),u(\tau))d\tau-f(t,x_{0}(t),v(t)))\\
((x_{m}^{\prime}(t)-x_{0}^{\prime}(t))+\int_{0}^{t}\Phi_{x}(t,\tau,x_{0}%
(\tau),u(\tau))(x_{m}(\tau)-x_{0}(\tau))d\tau\\
-f_{x}(t,x_{0}(t),v(t))(x_{m}(t)-x_{0}(t)))dt
\end{multline*}
and
\[
\varphi^{\prime}(x_{m})(x_{m}-x_{0})-\varphi^{\prime}(x_{0})(x_{m}%
-x_{0})=\left\Vert x_{m}-x_{0}\right\Vert _{AC_{0}^{2}}^{2}+\sum
\nolimits_{i=1}^{14}\psi_{i}(x_{m})
\]
where%
\[
\psi_{1}(x_{m})=\int\nolimits_{0}^{1}(\int_{0}^{t}\Phi(t,\tau,x_{m}%
(\tau),u(\tau))d\tau-\int_{0}^{t}\Phi(t,\tau,x_{0}(\tau),u(\tau))d\tau
)((x_{m}^{\prime}(t)-x_{0}^{\prime}(t))dt
\]%
\[
\psi_{2}(x_{m})=\int\nolimits_{0}^{1}(f(t,x_{0}(t),u(t))-f(t,x_{m}%
(t),u(t)))(x_{m}^{\prime}(t)-x_{0}^{\prime}(t))dt
\]%
\[
\psi_{3}(x_{m})=\int\nolimits_{0}^{1}x_{m}^{\prime}(t)\int_{0}^{t}\Phi
_{x}(t,\tau,x_{m}(\tau),u(\tau))(x_{m}(\tau)-x_{0}(\tau))d\tau dt
\]%
\[
\psi_{4}(x_{m})=\int\nolimits_{0}^{1}x_{0}^{\prime}(t)\int_{0}^{t}\Phi
_{x}(t,\tau,x_{0}(\tau),u(\tau))(x_{m}(\tau)-x_{0}(\tau))d\tau dt
\]%
\[
\psi_{5}(x_{m})=\int\nolimits_{0}^{1}(\int_{0}^{t}\Phi(t,\tau,x_{m}%
(\tau),u(\tau))d\tau\int_{0}^{t}\Phi_{x}(t,\tau,x_{m}(\tau),u(\tau
))(x_{m}(\tau)-x_{0}(\tau))d\tau)dt
\]%
\[
\psi_{6}(x_{m})=-\int\nolimits_{0}^{1}(\int_{0}^{t}\Phi(t,\tau,x_{0}%
(\tau),u(\tau))d\tau\int_{0}^{t}\Phi_{x}(t,\tau,x_{0}(\tau),u(\tau
))(x_{m}(\tau)-x_{0}(\tau))d\tau)dt
\]%
\[
\psi_{7}(x_{m})=-\int\nolimits_{0}^{1}(f(t,x_{m}(t),v(t))\int_{0}^{t}\Phi
_{x}(t,\tau,x_{m}(\tau),u(\tau))(x_{m}(\tau)-x_{0}(\tau))d\tau)dt
\]%
\[
\psi_{8}(x_{m})=\int\nolimits_{0}^{1}(f(t,x_{0}(t),v(t))\int_{0}^{t}\Phi
_{x}(t,\tau,x_{0}(\tau),u(\tau))(x_{m}(\tau)-x_{0}(\tau))d\tau)dt
\]%
\[
\psi_{9}(x_{m})=-\int\nolimits_{0}^{1}x_{m}^{\prime}(t)f_{x}(t,x_{m}%
(t),v(t))(x_{m}(t)-x_{0}(t)))dt
\]%
\[
\psi_{10}(x_{m})=-\int\nolimits_{0}^{1}x_{0}^{\prime}(t)f_{x}(t,x_{0}%
(t),v(t))(x_{m}(t)-x_{0}(t)))dt
\]%
\[
\psi_{11}(x_{m})=-\int\nolimits_{0}^{1}\int_{0}^{t}\Phi(t,\tau,x_{m}%
(\tau),u(\tau))d\tau f_{x}(t,x_{m}(t),v(t))(x_{m}(t)-x_{0}(t)))dt
\]%
\[
\psi_{12}(x_{m})=\int\nolimits_{0}^{1}\int_{0}^{t}\Phi(t,\tau,x_{0}%
(\tau),u(\tau))d\tau f_{x}(t,x_{0}(t),v(t))(x_{m}(t)-x_{0}(t)))dt
\]%
\[
\psi_{13}(x_{m})=-\int\nolimits_{0}^{1}f(t,x_{m}(t),v(t))f_{x}(t,x_{m}%
(t),v(t))(x_{m}(t)-x_{0}(t)))dt
\]%
\[
\psi_{14}(x_{m})=\int\nolimits_{0}^{1}f(t,x_{0}(t),v(t))f_{x}(t,x_{0}%
(t),v(t))(x_{m}(t)-x_{0}(t)))dt
\]
We shall show that $\varphi$ satisfies (PS) condition. Indeed, if $(x_{m})$ is
a (PS) sequence for $\varphi$, then the coercivity of $\varphi$ implies its
boundedness. Consequently, there exists a subsequence $(x_{m_{k}})$ which is
weakly convergent in $AC_{0}^{2}$ to some $x_{0}$\thinspace(so, $x_{m_{k}%
}\rightrightarrows x_{0}$ uniformly on $[0,1]$ and $x_{m_{k}}^{\prime
}\rightharpoonup x_{0}^{\prime}$ weakly in $L^{2}(I,\mathbb{R}^{n})$).

First, we shall show that $\psi_{i}(x_{m_{k}})\underset{k\rightarrow\infty
}{\longrightarrow}0$ for $i=1,...,14$.

Let us consider the first term $\psi_{1}(x_{m_{k}})$. From the Lebesgue
dominated convergence theorem it follows that
\[
\int_{0}^{t}(\Phi(t,\tau,x_{m_{k}}(\tau))-\Phi(t,\tau,x_{0}(\tau
)))d\tau\underset{m\rightarrow\infty}{\longrightarrow}0
\]
for $t\in\lbrack0,1]$ a.e. Moreover (cf. (A$_{2}$)),%
\begin{multline*}
\left\vert \int_{0}^{t}(\Phi(t,\tau,x_{m_{k}}(\tau))-\Phi(t,\tau,x_{0}%
(\tau)))d\tau\right\vert \\
\leq2\int_{0}^{t}(a(t,\tau)\left\vert x_{m_{k}}(\tau)\right\vert
+b(t,\tau)\omega(\left\vert u(\tau)\right\vert ))d\tau\leq2\int_{0}%
^{t}(a(t,\tau)M+b(t,\tau))d\tau
\end{multline*}
where $M>0$ is such that%
\[
\left\vert x_{m_{k}}(\tau)\right\vert \leq M,\ \tau\in\lbrack0,1],\ k=0,1,...
\]
Since the function%
\[
\lbrack0,1]\ni t\longmapsto2%
{\textstyle\int\nolimits_{0}^{t}}
(a(t,\tau)M+b(t,\tau)\omega(\left\vert u(\tau)\right\vert ))d\tau\in\mathbb{R}%
\]
belongs to $L^{2}([0,1],\mathbb{R}^{n})$, using once again the Lebesgue
dominated convergence theorem we assert that
\[
\int_{0}^{\cdot}(\Phi(\cdot,\tau,x_{m_{k}}(\tau))-\Phi(\cdot,\tau,x_{0}%
(\tau)))d\tau\underset{m\rightarrow\infty}{\longrightarrow}0
\]
in $L^{2}([0,1],\mathbb{R}^{n})$. Consequently, $\psi_{1}(x_{m_{k}})$ as a
scalar product in $L^{2}([0,1],\mathbb{R}^{n})$ of the functions $x_{m_{k}%
}^{\prime}(\cdot)-x_{0}^{\prime}(\cdot)$ and $%
{\textstyle\int\nolimits_{0}^{\cdot}}
(\Phi(\cdot,\tau,x_{m_{k}}(\tau))-\Phi(\cdot,\tau,x_{0}(\tau)))d\tau$ tends to
$0$ as $k\rightarrow\infty$. Similarly, using the growth condition on $f$ we
assert that $\psi_{2}(x_{m_{k}})\rightarrow0$. Convergence $\psi_{i}(x_{m_{k}%
})\rightarrow0$ for $i=3,...,14$ follows from the uniform convergence
$x_{m_{k}}\rightrightarrows x_{0}$.

Since $\varphi^{\prime}(x_{0})$ is linear and continuous functional on
$AC_{0}^{2}$, convergence $\varphi^{\prime}(x_{0})(x_{m_{k}}-x_{0}%
)\rightarrow0$ follows directly from the weak convergence $x_{m_{k}%
}\rightharpoonup x_{0}$ in $AC_{0}^{2}$. Convergence $\varphi^{\prime
}(x_{m_{k}})(x_{m_{k}}-x_{0})\rightarrow0$ follows from the estimation%
\[
\left\vert \varphi^{\prime}(x_{m_{k}})(x_{m_{k}}-x_{0})\right\vert
\leq\left\Vert \varphi^{\prime}(x_{m_{k}})\right\Vert _{\mathcal{L}(AC_{0}%
^{2},\mathbb{R})}\left\Vert x_{m_{k}}-x_{0}\right\Vert _{AC_{0}^{2}},
\]
boundedness of the sequence $(x_{m_{k}})$ and convergence $\varphi^{\prime
}(x_{m_{k}})\rightarrow0$.

So, all assumptions of the global implicit function theorem are satisfied.
Consequently, for any $(u,v)\in L^{\infty}(J,\mathbb{R}^{m})\times L^{\infty
}(J,\mathbb{R}^{r})$ there exists a unique solution $x_{u,v}\in AC_{0}^{2}$ of
the equation (\ref{cs}) and the mapping%
\[
\lambda:L^{\infty}(J,\mathbb{R}^{m})\times L^{\infty}(J,\mathbb{R}^{r}%
)\ni(u,v)\longmapsto x_{u,v}\in AC_{0}^{2}%
\]
is continuously differentiable in Gateaux (equivalently, in Frechet) sense on
$L^{\infty}(J,\mathbb{R}^{m})\times L^{\infty}(J,\mathbb{R}^{r})$ and the
differential $\lambda^{\prime}(u,v)$ at a point $(u,v)\in L^{\infty
}(J,\mathbb{R}^{m})\times L^{\infty}(J,\mathbb{R}^{r})$ is the following
\[
\lambda^{\prime}(u,v):L^{\infty}(J,\mathbb{R}^{m})\times L^{\infty
}(J,\mathbb{R}^{r})\ni(f,g)\longmapsto z_{f,g}\in AC_{0}^{2}%
\]
where $z_{f,g}$ is such that%
\begin{multline*}
z_{f,g}^{\prime}(t)+\int_{0}^{t}\Phi_{x}(t,\tau,x_{u,v}(\tau),u(\tau
))z_{f,g}(\tau)d\tau-f_{x}(t,x_{u,v}(t),v(t))z_{f,g}(t)\\
=-\int_{0}^{t}\Phi_{u}(t,\tau,x_{u,v}(\tau),u(\tau))f(\tau)d\tau
+f_{v}(t,x_{u,v}(t),v(t))g(t)
\end{multline*}
a.e. on $J$.

\section{Appendix}

Let us consider the following control system
\begin{equation}
x^{\prime}(t)+\int_{0}^{t}\Psi(t,\tau,x(\tau))d\tau=g(t,x(t)),\ t\in J\text{
a.e.,} \label{psig}%
\end{equation}
where $\Psi:P_{\Delta}\times\mathbb{R}^{n}\rightarrow\mathbb{R}^{n}$,
$g:J\times\mathbb{R}^{n}\rightarrow\mathbb{R}^{n}$, $x\in AC_{0}^{2}$. On the
functions $\Psi$, $g$ we assume that

\begin{itemize}
\item[$\cdot$] $\Psi(\cdot,\cdot,x)$ is measurable on $P_{\Delta}$ for any
$x\in\mathbb{R}^{n}$ and%
\[
\left\vert \Psi(t,\tau,x_{1})-\Psi(t,\tau,x_{2})\right\vert \leq M\left\vert
x_{1}-x_{2}\right\vert
\]
for $(t,\tau)\in P_{\Delta}$ a.e., $x_{1}$, $x_{2}\in\mathbb{R}^{n}$, where
$M>0$ is some constant

\item[$\cdot$] $\Psi(\cdot,\cdot,0)\in L^{2}(P_{\Delta},\mathbb{R}^{n})$

\item[$\cdot$] $g(\cdot,x)$ is measurable on $J$ for any $x\in\mathbb{R}^{n}$
and%
\[
\left\vert g(t,x_{1})-g(t,x_{2})\right\vert \leq L\left\vert x_{1}%
-x_{2}\right\vert
\]
for $t\in J$ a.e., $x_{1}$, $x_{2}\in\mathbb{R}^{n}$, where $L>0$ is some constant

\item[$\cdot$] $g(\cdot,0)\in L^{2}(J,\mathbb{R}^{n})$
\end{itemize}

It is easy to see that the existence of a solution $x$ to system (\ref{psig})
in the space $AC_{0}^{2}$ is equivalent to the existence of a solution $l$ to
system%
\begin{equation}
l(t)+\int_{0}^{t}\Psi(t,\tau,\int_{0}^{\tau}l(s)ds)d\tau=g(t,\int_{0}%
^{t}l(\tau)d\tau),\ t\in J\text{ a.e.,}%
\end{equation}
in the space $L^{2}(J,\mathbb{R}^{n})$; in such a case $x^{\prime}=l$. To
prove that the above system has a unique solution in $L^{2}(J,\mathbb{R}^{n})$
we shall show that the operator%
\begin{equation}
T:L^{2}(J,\mathbb{R}^{n})\ni l\longmapsto g(t,\int_{0}^{t}l(\tau)d\tau
)-\int_{0}^{t}\Psi(t,\tau,\int_{0}^{\tau}l(s)ds)d\tau\in L^{2}(J,\mathbb{R}%
^{n}) \label{operator}%
\end{equation}
is contracting.

\begin{theorem}
There exists a unique fixed point of the operator $T$ and, consequently,
system (\ref{psig}) has a unique solution in $AC_{0}^{2}$.
\end{theorem}

\begin{proof}
We shall show that there exists a positive integer $k$ such that the
operators
\[
T_{g}:L^{2}(J,\mathbb{R}^{n})\ni l\longmapsto g(t,\int_{0}^{t}l(\tau)d\tau)\in
L^{2}(J,\mathbb{R}^{n})
\]%
\[
T_{\Psi}:L^{2}(J,\mathbb{R}^{n})\ni l\longmapsto\int_{0}^{t}\Psi(t,\tau
,\int_{0}^{\tau}l(s)ds)d\tau\in L^{2}(J,\mathbb{R}^{n})
\]
are contracting with a constants $\xi_{1}$, $\xi_{2}\in(0,1/2)$, respectively,
if $L^{2}(J,\mathbb{R}^{n})$ is considered with the Bielecki norm%
\[
\left\Vert l\right\Vert _{k}=(\int_{0}^{1}e^{-kt}\left\vert l(t)\right\vert
^{2}dt)^{1/2},\ l\in L^{2}(J,\mathbb{R}^{n})\text{.}%
\]
Indeed, let us fix $k\in\mathbb{N}$. We have
\begin{multline*}
\left\Vert T_{g}(l_{1})-T_{g}(l_{2})\right\Vert _{k}^{2}\leq\int_{0}%
^{1}e^{-kt}\left\vert T_{g}(l_{1})-T_{g}(l_{2})\right\vert ^{2}dt\\
\leq L^{2}\int_{0}^{1}e^{-kt}\left\vert \int_{0}^{t}\left\vert l_{1}%
(\tau)-l_{2}(\tau)\right\vert d\tau\right\vert ^{2}dt\leq L^{2}\int_{0}%
^{1}e^{-kt}\int_{0}^{t}\left\vert l_{1}(\tau)-l_{2}(\tau)\right\vert ^{2}d\tau
dt\\
=L^{2}(-\frac{1}{k}e^{-kt}\int_{0}^{1}\left\vert l_{1}(\tau)-l_{2}%
(\tau)\right\vert ^{2}d\tau+\frac{1}{k}\int_{0}^{1}e^{-kt}\left\vert
l_{1}(t)-l_{2}(t)\right\vert ^{2}dt)\leq\frac{L^{2}}{k}\left\Vert l_{1}%
-l_{2}\right\Vert _{k}^{2}%
\end{multline*}
for $l_{1}$, $l_{2}\in L^{2}(J,\mathbb{R}^{n})$.

Similarly,%
\begin{multline*}
\left\Vert T_{\Psi}(l_{1})-T_{\Psi}(l_{2})\right\Vert _{k}^{2}\leq\int_{0}%
^{1}e^{-kt}\left\vert T_{\Psi}(l_{1})-T_{\Psi}(l_{2})\right\vert ^{2}dt\\
\leq\int_{0}^{1}e^{-kt}\left\vert \int_{0}^{t}\left\vert \Psi(t,\tau,\int
_{0}^{\tau}l_{1}(s)ds)-\Psi(t,\tau,\int_{0}^{\tau}l_{2}(s)ds)\right\vert
d\tau\right\vert ^{2}dt\\
\leq M^{2}\int_{0}^{1}e^{-kt}\left\vert \int_{0}^{t}\int_{0}^{\tau}\left\vert
l_{1}(s)-l_{2}(s)\right\vert dsd\tau\right\vert ^{2}dt\\
\leq M^{2}\int_{0}^{1}e^{-kt}\left\vert \int_{0}^{t}\left\vert l_{1}%
(\tau)-l_{2}(\tau)\right\vert d\tau\right\vert ^{2}dt\leq\frac{M^{2}}%
{k}\left\Vert l_{1}-l_{2}\right\Vert _{k}^{2}%
\end{multline*}
for $l_{1}$, $l_{2}\in L^{2}(J,\mathbb{R}^{n})$.

So, to end the proof it is sufficient to choose $k$ such that $\max
\{\sqrt{\frac{L^{2}}{k}},\sqrt{\frac{M^{2}}{k}}\}<1/2$.
\end{proof}

\noindent\textbf{Acknowledgement}. The project was financed with funds of
National Science Centre, granted on the basis of decision DEC-2011/01/B/ST7/03426.

\end{document}